\documentclass[10pt]{article}
\usepackage{latexsym,amssymb,amsmath,amsthm,amsfonts}
\usepackage{graphicx}
\usepackage{enumerate}
\usepackage{color}

\theoremstyle{plain}
\newtheorem{thm}{Theorem}[section]
\newtheorem{prop}[thm]{Proposition}
\newtheorem{lem}[thm]{Lemma}

\theoremstyle{definition}

\theoremstyle{remark}

\numberwithin{equation}{section}

\newcommand{\ud}{\mathrm{d}}
\newcommand{\RR}{\mathbb{R}}
\newcommand{\f}{\frac}

\newcommand{\pp}[2]{\frac{\partial{#1}}{\partial{#2}}}

\newcommand{\pppp}[4]%
  {\frac{\partial^3{#1}}{\partial{#2}\partial{#3}\partial{#4}}}

\newcommand{\p}{\phi}

\newcommand{\pab}{\alpha\phi\left(\frac{\beta}{\alpha}\right)}

\newcommand{\sq}{\frac{(\alpha+\beta)^2}{\alpha}}

\renewcommand{\a}{\alpha}
\renewcommand{\b}{\beta}
\newcommand{\ab}{(\alpha,\beta)}
\newcommand{\ta}{\tilde\alpha}
\newcommand{\tb}{\tilde\beta}

\newcommand{\ba}{\bar\alpha}
\newcommand{\bb}{\bar\beta}
\newcommand{\bra}{\breve{\alpha}}

\newcommand{\aij}{a_{ij}}

\newcommand{\bi}{b_i}
\newcommand{\bj}{b_j}

\newcommand{\bij}{b_{i|j}}
\newcommand{\taij}{\tilde a_{ij}}

\newcommand{\tbij}{\tilde b_{i|j}}

\newcommand{\baij}{\bar a_{ij}}

\newcommand{\bbij}{\bar b_{i|j}}

\newcommand{\G}{{}^\alpha G^i}
\newcommand{\tG}{{}^{\tilde\alpha}G^i}

\newcommand{\bG}{{}^{\bar\alpha}G^i}

\newcommand{\Ric}{{}^\alpha\mathrm{Ric}}

\newcommand{\tRic}{{}^{\tilde{\alpha}}{\mathrm{Ric}}}
\newcommand{\bRic}{{}^{\bar{\alpha}}{\mathrm{Ric}}}
\newcommand{\brRic}{{}^{\breve{\alpha}}\mathrm{Ric}}

\begin{document}
\title{On Einstein Square Metrics
}
\author{Zhongmin Shen\footnote{supported in part by a NSF grant~(DMS-0810159) and a NSFC grant (No.11171297)} and Changtao Yu\footnote{supported by a NSFC grant(No.11026097)}}
%\date{}
\maketitle

\begin{abstract}
In this paper, we study an important class of Finsler metrics --- square metrics. We give two expressions of such metrics in terms of a Riemannian metric and a $1$-form. We show that Einstein square metrics can be classified up to the classification of Einstein Riemannian metrics.

\end{abstract}

\section{Introduction}

One of important problems in Finsler geometry is to characterize and construct Finsler metrics with constant flag curvature or Ricci curvature, the later are called Einstein metrics. Many valuable results have been achieved, most of which are related to a special class of Finsler metrics named $(\a,\b)$-metrics due to its computability.

$(\a,\b)$-metrics, given in the following form
$$F=\a\p(s),\quad s=\f{\b}{\a},$$
are defined by a Riemannian metric $\a$, an 1-form $\b$ and a smooth function $\phi(s)$. The simplest and most important $(\a,\b)$-metrics are Randers metrics $F=\a+\b$, which can be rewritten via the navigation problem as
\begin{eqnarray}\label{NPexpression}
F=\f{\sqrt{(1-|W|^2_h)h^2+(W^\flat)^2}}{1-|W|_h^2}-\f{W^\flat}{1-|W|^2_h},
\end{eqnarray}
where $h$ and $W$ are a Riemannian metric and a vector field respectively. They are related to the original data by
$$h=\sqrt{1-b^2}\sqrt{\a^2-\b^2},\quad W^\flat=-(1-b^2)\b.$$
Notice that the Randers metric $F=\a+\b$ is regular if and only if $b:=\|\b\|_\a<1$,  or equivalently, $|W|_h<1$.

In 2004, D. Bao, C. Robles and the first author classified Randers metrics with constant flag curvature (\cite{db-robl-szm-zerm}) using the navigation expression in (\ref{NPexpression}):~$F=\a+\b$ is of constant flag curvature if and only if $h$ is of constant sectional curvature and $W$ is a homothetic  vector field with respect to $h$. i.e.,
$$W_{i|j}+W_{j|i}=ch_{ij}$$
for some constant $c$.

Also with the help of the navigation problem, D. Bao and C. Robles give a characterization for Einstein metrics of Randers type\cite{db-robl-onri}:~$F=\a+\b$ is an Einstein metric if and only if $h$ is an Einstein metric and $W$ is a homothetic vector field with respect to $h$.

Besides Randers metrics, there are other interesting $\ab$-metrics. For example, the following metric
\begin{eqnarray}\label{berwald}
F=\frac{(\sqrt{(1-|x|^2)|y|^2+\langle x,y\rangle^2}+\langle
x,y\rangle)^2}{(1-|x|^2)^2\sqrt{(1-|x|^2)|y|^2+\langle
x,y\rangle^2}},
\end{eqnarray}
constructed by L. Berwald in 1929, is projectively flat on the unit ball $\mathbb B^n$ with constant flag curvature $K=0$\cite{berw}. Berwald's metric can be expressed as
\begin{eqnarray}\label{square}
F=\f{(\a+\b)^2}{\a}
\end{eqnarray}
where
$$\a=\f{\sqrt{(1-|x|^2)|y|^2+\langle x,y\rangle^2}}{(1-|x|^2)^2},\quad\b=\f{\langle x,y\rangle}{(1-|x|^2)^2}.$$
An $(\alpha,\beta)$-metric in the form (\ref{square}) is called a {\it square metric}.

In 2007, B. Li and the first author classified projectively flat $\ab$-metrics with constant flag curvature\cite{lb-szm-onac}. The result shows that such Finsler metrics of non-Randers type are essentially square metrics. Later on, L. Zhou proves that any square metric with constant flag curvature must be locally projectively flat\cite{zlf}.

It seems that square metrics  play a particular role in Finsler geometry, just as Randers metrics.
The main purpose of this paper is to determine the structure of Einstein square metrics $F=\frac{(\a+\b)^2}{\a}$. Firstly, we have the following result.

\begin{thm}\label{main2}
Let $F=\sq$ be a Finsler metric on an $n$-dimensional manifold $M$. Then $F$ is an Einstein metric if and only if it is Ricci flat and
\begin{eqnarray}
\Ric&=&c^2(1-b^2)^2\left\{-[5(n-1)+2(2n-5)b^2]\a^2+6(n-2)\b^2\right\},\label{eq1}\\
\bij&=&c(1-b^2)\left\{(1+2b^2)a_{ij}-3b_ib_j\right\},\label{eq2}
\end{eqnarray}
where $c$ is a constant number.
\end{thm}

Theorem \ref{main2} under an additional condition of $\beta$ being closed is proved in \cite{zrrr}. Theorem \ref{main2} under an additional condition of $F$ being of Douglas type is proved independently by Cheng-Tian\cite{cxy-tyf} and  Sevim-Shen-Zhao 
\cite{sevim-szm-zll}
 independently. Recently,  Chen-Shen-Zhao find explicit solutions of (\ref{eq1}) and (\ref{eq2}) in terms  of $(n-1)$-dimensional Riemannian Einstein metrics   \cite{cb-szm-zll}.

Using suitable deformations of $\alpha$ and $\beta$, we can express a square metric $F=(\alpha+\beta)^2/\alpha$ using another set of Riemannian metric $\tilde{\alpha}$ and $1$-form  $\tilde{\beta}$   so that the square metric is an Einstein metric if and only if $\tilde{\alpha}$ is an Einstein metric and $\tilde{\beta}$ is conformal (see Theorem \ref{main} below). This is very similar to Randers metrics. 
Then we can also determine the local structure of Einstein square metrics (see Proposition \ref{c} below).

\begin{thm}\label{main}
Let $F=\sq$ be a Finsler metric on an $n$-dimensional manifold $M$. Then the following are equivalent:
\begin{enumerate}[(1)]
\item $F$ is an Einstein metric;
\item The Riemannian metric $\ta:=(1-b^2)\a$ and the $1$-form $\tb:=\sqrt{1-b^2}\b$ satisfy
\begin{eqnarray}\label{tildeab}
\tRic=-(n-1)c^2\ta,\qquad\tbij=c\sqrt{1+\tilde b^2}\taij,
\end{eqnarray}
where $c$ is a constant number, $\tilde b=\|\tb\|_{\ta}$, and $\tbij$ is the covariant derivation of $\tb$ with respect to $\ta$. In this case, $F$ is given in the following form
\begin{eqnarray}\label{expression1}
F=\f{(\sqrt{1+\tilde b^2}\ta+\tb)^2}{\ta}
\end{eqnarray}
with $(1+\tilde{b}^2)(1-b^2) =1$.
\item The Riemannian metric $\ba:=(1-b^2)^\frac{3}{2}\sqrt{\a^2-\b^2}$ and the $1$-form $\bb:=(1-b^2)^2\b$ satisfy
\begin{eqnarray}\label{barab}
\bRic=0,\qquad\bbij=c\baij,
\end{eqnarray}
where $c$ is a constant number, $\bar b=\|\bb\|_{\ba}$ and $\bbij$ is the covariant derivation of $\bb$ with respect to $\ba$. In this case, $F$ is given in the following form
\begin{eqnarray}\label{expression2}
F=\f{(\sqrt{(1-\bar b^2)\ba^2+\bb^2}+\bb)^2}{(1-\bar b^2)^2\sqrt{(1-\bar b^2)\ba^2+\bb^2}}
\end{eqnarray}
with $\bar{b}= b$.
\end{enumerate}
\end{thm}

In particular, if we take $\ba=|y|$ and $\bb=\langle x,y\rangle$, then the Finsler metric $F$ given by (\ref{expression2}) is just the Berwald's metric (\ref{berwald}).

It's worth mentioning that the expressions (\ref{NPexpression}), (\ref{expression1}) and (\ref{expression2}) are all in the following special form
\begin{eqnarray}\label{gab}
F=\a\phi(b^2,s),\quad s=\f{\b}{\a}
\end{eqnarray}
for some smooth function $\phi(b^2,s)$. This kind of Finsler metrics belongs to a new class of Finsler metrics called {\it  general $\ab$-metrics}, which is proposed by the second author as a generalization of Randers metrics from the geometric point of view\cite{yct-zhm-onan}. General $\ab$-metrics include all the $\ab$-metrics naturally.

It is easy to see that the corresponding functions of (\ref{NPexpression}), (\ref{expression1}) and (\ref{expression2}) are given by (\ref{gab}) with
\begin{eqnarray*}
\phi(b^2,s)&=&\f{\sqrt{1-b^2+s^2}}{1-b^2}-\f{s}{1-b^2},\\
\phi(b^2,s)&=&(\sqrt{1+b^2}+s)^2
\end{eqnarray*}
and
$$\phi(b^2,s)=\f{(\sqrt{1-b^2+s^2}+s)^2}{(1-b^2)^2\sqrt{1-b^2+s^2}},$$
respectively.
It is marvelous that all these functions $\phi=\phi(b^2, s)$  satisfy a simple partial differential equation
$$\phi_{22}=2(\phi_1-s\phi_{12}),$$
where $\phi_1$ means the derivation of $\phi$ with respect to the first variable $b^2$ \cite{yct-zhm-onan}.

In our opinion, general $\ab$-metrics bring great freedom, which makes it possible to seek the most appropriate description for some specific $\ab$-metrics or general $\ab$-metrics. The expressions (\ref{expression1}) and (\ref{expression2}) are more complicated than (\ref{square}) in algebra form, but they have the advantage of clearly illuminating the underlying geometry, just as the navigation expression (\ref{NPexpression}) for Randers metrics. In 2008, the second author proved in his doctoral dissertation that the Finsler metric $F=\sq$ is locally projectively flat if and only if the corresponding Riemannian metric in (\ref{expression1}) or (\ref{expression2}) is locally projectively flat and the 1-form is closed and conformal with respect to the Riemannian metric. We believe that the expression (\ref{expression1}), especially (\ref{expression2}), are tailored for the $\ab$-metric $F=\sq$, although it seems that there is not a good physical or geometric explanation for this kind of $\ab$-metrics.

Finally, with an argument similar to that in Section 3, we can provide a new description for such kind of $\ab$-metrics with constant flag curvature.

\begin{thm}
The Finsler metric $F=\sq$ is of constant flag curvature if and only if under the expression (\ref{expression2}) of $F$, $\ba$ is locally Euclidean, $\bb$ is closed and homothety with respect to $\ba$. In a suitable local coordinates, $F$ can be expressed by
\begin{equation}
 F = \frac{(\sqrt{ (1-|\bar{x}|^2) |y|^2 + \langle \bar{x}, y\rangle^2   } +\langle \bar{x}, y\rangle )^2 }{ (1-|\bar{x}|^2)^2 \sqrt{ (1-|\bar{x}|^2) |y|^2 + \langle \bar{x}, y \rangle)^2     }},\label{Bmetric}
\end{equation}
where $\bar{x} := cx + a$
for some constant number $c$ and constant vector $a$. In particular, $F$ must be locally projectively flat with zero flag curvature.

\end{thm}

\section{Preliminaries}

Given a Finsler metric $F$ on a $n$-dimensional smooth manifold $M$. There is a global vector field $G$ on the slit tangent bundle $TM\backslash\{0\}$ which is called  a {\it spray}. In local coordinates, $G=y^i\frac{\partial}{\partial x^i}-2G^i\frac{\partial}{\partial y^i}$ where
$$G^i=\f{1}{4}g^{il}\left\{[F^2]_{x^ky^l}y^k-[F^2]_{x^l}\right\}$$
are geodesic spray coefficients. Here $(g^{ij})$ is the inverse of the Hessian $g_{ij}=\frac{1}{2}[F^2]_{y^iy^j}$.

By Berwald's formulae, the components of the Riemann curvature tensor of $F$ are given by
\begin{eqnarray*}
R^i{}_k=2\pp{G^i}{x^k}-\f{\partial^2 G^i}{\partial x^my^k}{y^m}+2G^m\f{\partial^2G^i}{\partial y^m\partial y^k}-\pp{G^i}{y^m}\pp{G^m}{y^k}.
\end{eqnarray*}
A Finsler metric is said to be of constant flag curvature $K$ if and only if
$$R^i{}_k=K(\delta^i{}_kF^2-y^ig_{kj}y^j).$$
The Ricci curvature $Ric$ of $F$ is the trace of the flag curvature tensor. A Finsler metric is called an Einstein metric if $Ric=(n-1)cF^2$ for some scalar function $c=c(x)$ on $M$. The Schur Lemma for Ricci curvature is true for Randers metrics\cite{db-robl-onri}, but it is still a conjecture in general case.

Besides the Riemann curvature tensor, there is a non-Riemannian quantity introduced by J. Douglas given by
$$D_j{}^i{}_{kl}=\left(G^i-\f{1}{n+1}\f{\partial G^m}{\partial y^m}y^i\right)_{y^jy^ky^l}.$$
A Finsler metric is said to be Douglas if $D_j{}^i{}_{kl}=0$. Douglas metrics are those with projectively affine geodesics.

Given an $\ab$-metric $F=\pab$. Let $b_{i|j}$ denote the coefficients of the covariant derivative of
$\b=b_iy^i$ with respect to $\a$, and
$$r_{ij}=\frac{1}{2}(b_{i|j}+b_{j|i}),~s_{ij}=\frac{1}{2}(b_{i|j}-b_{j|i}),
~r_{00}=r_{ij}y^iy^j,$$
$$s_{i0}=s_{ij}y^j,~s^i{}_0=a^{ij}s_{j0},~r_0=r_{ij}b^iy^j,~s_0=s_{ij}b^iy^j.$$
It is easy to see that $\b$ is closed if and only if $s_{ij}=0$.

According to \cite{css-szm-riem}, the geodesic spray coefficients $G^i$ are given by
\begin{eqnarray*}
G^i= {}^\a G^i+\a Q s^i{}_0+\a^{-1}\Theta(-2\a Q s_0+r_{00})y^i+\Psi(-2\a Q s_0+r_{00})b^i,
\end{eqnarray*}
where
\begin{eqnarray*}
Q=\frac{\p'}{\p-s\p'},~\Theta=\frac{(\p-s\p')\p'-s\p\p''}{2\p\big(\p-s\p'+(b^2-s^2)\p''\big)},~
\Psi=\frac{\p''}{2\big(\p-s\p'+(b^2-s^2)\p''\big)},
\end{eqnarray*}
and
\begin{eqnarray}\label{ag}
{}^\a G^i=\f{1}{4}a^{il}\left\{[\a^2]_{x^ky^l}y^k-[\a^2]_{x^l}\right\}
\end{eqnarray}
are the geodesic spray coefficients of the Riemannian metric $\a$.

Furthermore, if the geodesic spray coefficients of a Finsler metric $F$ are given by
$$G^i={}^\a G^i+Q^i,$$
then the Riemann curvature  of $F$ are related to that of $\a$ and given by
$$R^i{}_j={}^\a R^i{}_j+2Q^i{}_{|j}-y^mQ^i{}_{|m.j}+2Q^mQ^i{}_{.m.j}-Q^i{}_{.m}Q^m{}_{.j},$$
where ``$|$" and  ``$.$" denote the horizontal covariant derivative and vertical covariant derivative with respect to $\a$ respectively. So
\begin{eqnarray}\label{riccichange}
\mathrm{Ric}=\Ric+2Q^i{}_{|i}-y^jQ^i{}_{|j.i}+2Q^jQ^i{}_{.j.i}-Q^i{}_{.j}Q^j{}_{.i},
\end{eqnarray}
where $\Ric$ is the Ricci curvature of $\a$.

Using a Maple program, one can obtain formulas for the Riemann curvature and the Ricci curvature \cite{cxy-szm-tyf}. Using those formulas, one can obtain equations on $\alpha$ and $\beta$ that characterize Einstein square metrics. However, due to the complexity of the formulas, people only consider some special cases when the square metric is of Douglas type or when the $1$-form of the square metric is closed.

\begin{thm}[\cite{cxy-tyf,sevim-szm-zll,zrrr}]\label{original}
Let $F=\sq$ be a Finsler metric on an $n$-dimensional manifold $M$. Suppose that $F$ is of Doulas type $(n \geq 3$) or $\beta$ is closed. Then $F$ is an Einstein metric if and only if it is Ricci flat and
\begin{eqnarray}
\Ric&=&\tau^2\left\{-[5(n-1)+2(2n-5)b^2]\a^2+6(n-2)\b^2\right\},\label{RicRic}\\
\bij&=&\tau\left\{(1+2b^2)a_{ij}-3b_ib_j\right\},\label{bij}\\
\tau_i&=&-2\tau^2b_i,\label{tau}
\end{eqnarray}
where $\tau=\tau(x)$ is a function on $M$.
\end{thm}

It is known that for a square metric $F=(\alpha+\beta)^2/\alpha$ in dimension $n\geq 3$, if it is of Douglas type, then $\beta$ is closed.
Sevim-Shen-Zhao and Cheng-Tian independently show that for a square metric $F=(\alpha+\beta)^2/\alpha$ of Douglas  type in dimension $n\geq 3$,  it is Einstein if and only if   (\ref{RicRic})-(\ref{tau}) are satisfied. 
 Meanwhile, Zohrehvand-Rezaii prove that if $\beta$ is closed, then $F$ is Einstein if and only if (\ref{RicRic})-(\ref{tau}) are satisfied.  In this case, no restriction on $n$ is needed and $F$ is of Douglas type.   Thus these two versions  are equivalent at least in dimension $n\geq 3$.

In this paper, we make the following simple observation  that in any dimension, if a square metric $F=(\alpha+\beta)^2/\alpha$ is Einstein, then the $1$-form $\beta$ must be closed (Lemma \ref{closed} below). Therefore, the additional conditions in Theorem \ref{original} can  be dropped.

\begin{lem}\label{closed}
If the Finsler metric $F=\sq$ is an Einstein metric, then $\b$ is closed.
\end{lem}
\begin{proof}
By the proof of Lemma 4.2 in \cite{zlf}, we can see that if $F=\sq$ has constant Ricci curvature, then
$\b$ must satisfy
\begin{eqnarray}\label{sij}
s^k{}_0s_{k0}=0,\quad r_{ij}=\tau\{ (1+2b^2)a_{ij}-3 b_i b_j\},
\end{eqnarray}
where $\tau$ is a smooth function on the manifold. The first condition means $s_{k0}$~(as a vector)~is zero, so $\b$ is closed.
\end{proof}

We can  show that the function $\tau$ in Theorem \ref{original} depends only on the length of $\beta$.
\begin{lem}\label{tau}
Under the conditions of Theorem \ref{original},
$$\tau=c(1-b^2),$$
where $c$ is a constant.
\end{lem}
\begin{proof}
It is only have to verify that the derivative of $\frac{\tau}{1-b^2}$ is zero.
\end{proof}

With the above lemmas, we can prove Theorem \ref{main2}.

\begin{proof}[Proof of Theorem \ref{main2}]
Suppose that $F=\sq$ is an Einstein metric, then by Lemma \ref{closed} and the second equality of (\ref{sij}) we have
$$b_{i|j}=\tau\{ (1+2b^2)a_{ij}-3b_i b_j \} ,$$
which means that $F$ is of Douglas type \cite{LSShen}. So Theorem \ref{original} and Lemma \ref{tau} imply Theorem \ref{main2} immediately.
\end{proof}

\section{Proof of Theorem \ref{main}}

Suppose that $\ta=(1-b^2)\a$, then by (\ref{ag}) and (\ref{bij}) we have
$$\tG=\G+\tau(\a^2b^i-2\b y^i).$$
Let $\tilde Q^i=\tau(\a^2b^i-2\b y^i)$, then
\begin{eqnarray*}
\tilde Q^i{}_{|i}&=&\tau^2\left\{(n-2)(1+2b^2)\a^2-5(b^2\a^2-2\b^2)\right\},\\
y^j\tilde Q^i{}_{|j.i}&=&-2n\tau^2\left\{(1+2b^2)\a^2-5\b^2\right\},\\
\tilde Q^i{}_{.j}\tilde Q^j{}_{.i}&=&-4\tau^2\left\{2b^2\a^2-(n+2)\b^2\right\},\\
\tilde Q^j\tilde Q^i{}_{.j.i}&=&-2n\tau^2(b^2\a^2-2\b^2).
\end{eqnarray*}
So by (\ref{riccichange}) and Lemma \ref{tau} we obtain
$$\tRic=-(n-1)\tau^2\a^2=-(n-1)c^2\ta^2.$$
On the other hand, direct computations show that
\begin{eqnarray*}
\tbij&=&\tau\sqrt{1-b^2}\aij=\f{c}{\sqrt{1-b^2}}\taij,
\end{eqnarray*}
so $\tbij=c\sqrt{1+\tilde b^2}\taij$ because
$$(1-b^2)(1+\tilde b^2)=1.$$

Conversely, one can verify directly that if $\ta$ and $\tb$ satisfy (\ref{tildeab}), then the Riemannian metric $\a=(1+\tilde b^2)\ta$ and the 1-form $\b=\sqrt{1+\tilde b^2}\tb$ satisfy (\ref{eq1}) and (\ref{eq2}) since all the argument above are reversible.

\noindent 1~$\Leftrightarrow$~3

Suppose that $\ba=(1-b^2)^\frac{3}{2}\sqrt{\a^2-\b^2}$, then by (\ref{ag}) and (\ref{bij}) we have
$$\bG=\G+\tau(\a^2b^i-3\b y^i).$$
Let $\bar Q^i=\tau(\a^2b^i-3\b y^i)$, then
\begin{eqnarray*}
\bar Q^i{}_{|i}&=&\tau^2\left\{(n-3)(1+2b^2)\a^2-5(b^2\a^2-3\b^2)\right\},\\
y^j\bar Q^i{}_{|j.i}&=&-(3n+1)\tau^2\left\{(1+2b^2)\a^2-5\b^2\right\},\\
\bar Q^i{}_{.j}\bar Q^j{}_{.i}&=&-\tau^2\left\{12b^2\a^2-(9n+19)\right\},\\
\bar Q^j\bar Q^i{}_{.j.i}&=&-(3n+1)\tau^2(b^2\a^2-3\b^2).
\end{eqnarray*}
So by (\ref{riccichange}) and Lemma \ref{tau} we obtain
$$\bRic=0.$$
On the other hand, direct computations show that
\begin{eqnarray*}
\bbij&=&\tau(1-b^2)^2(\aij-\bi\bj)=c\baij.
\end{eqnarray*}
In this case,
$$\bar b=b.$$

Conversely, one can verify directly that if $\ba$ and $\bb$ satisfy (\ref{barab}), then the Riemannian metric $\a=(1-\bar b^2)^{-2}\sqrt{(1-\bar b^2)\ba^2+\bb^2}$ and the 1-form $\b=(1-\bar b^2)^{-2}\bb$ satisfy (\ref{eq1}) and (\ref{eq2}) since all the argument above are reversible.

\section{Examples}

Using Theorem \ref{main}, one can easily determine Einstein square metrics up to the classification of Einstein metrics with {non-negative} Ricci constant.  First we need the following

\begin{lem}[\cite{pete}]\label{petersen}
If there are smooth functions $f$ and $\lambda$ on a Riemannian manifold $(M,\alpha)$ such that $\mathrm{Hess}_\alpha f=\lambda\alpha^2$, then the Riemannian structure is a warped product {around any point where $\ud f\neq0$.} In particular, $\alpha^2=\ud t\otimes\ud t+(f'(t))^2\breve\alpha^2$ where $f$ depends on the parameter $t$.
\end{lem}

The Riemannian metric $\bar{\alpha}$ and the  $1$-form $\bar{\beta}$ in Theorem \ref{main} (3) can be determined.

\begin{prop}\label{c}
Let $\ba$ be a Riemannian metric on an $n$-dimensional manifold $M$ and $\bb$ an $1$-form ($n\geq 3$). Then $\ba$ is Ricci flat and $\bbij=c\baij$ for some constant number $c$ if and only if $\ba$ is locally a warped product metric on $\RR\times\breve{M}$ and $\bb$ is an $1$-form determined by the first factor $\RR$:
\begin{eqnarray}
\ba^2&=&\ud t\otimes\ud t+(ct+d)^2\breve{\alpha}^2,\\
\bb&=&(ct+d)\ud t
\end{eqnarray}
for some constant number $d$ satisfying $c^2+d^2\neq0$, where $\breve{\alpha}$ is an Einstein metric on $\breve{M}$ with
$$\brRic=(n-2)c^2\bra^2.$$
\end{prop}
\begin{proof}
Because $\bb$ is closed, we can assume that locally $\bb=\ud f\neq0$ for some smooth function. It is easy to see that the condition $\bbij=c\baij$ is equivalent to $\mathrm{Hess}_{\ba}f=c\ba^2$. By Lemma \ref{petersen}, $\ba=\ud t\otimes\ud t+h^2(t)\breve{\alpha}^2$ is locally a warped product metric on $\RR\times\breve{M}$ where $h(t)=f'(t)$. By taking the trace of the Riemannian curvature tensor the warped product metric $\ba$ given in \S~13.3 of \cite{bao-cheng-shen}, we have
$$\bRic={}^{\breve{\alpha}}\mathrm{Ric}-(n-1)\f{h''}{h}y^1y^1-\left[h''h+(n-2)(h')^2\right]\breve\alpha^2.$$
Because $\ba$ is Ricci flat, we obtain
$${}^{\breve{\alpha}}\mathrm{Ric}-(n-1)\f{h''}{h}y^1y^1-\left[h''h+(n-2)(h')^2\right]\breve\alpha^2=0.$$
Note that $\breve{\alpha}$ and ${}^{\breve{\alpha}}\mathrm{Ric}$ are both independent of $y^1$, the above equality means
$$h''(t)=0,\quad{}^{\breve{\alpha}}\mathrm{Ric}=(n-2)(h')^2\breve\alpha^2.$$
On the other hand, a direct computation shows that the covariant derivation of $\bb=h(t)\,\ud t$ with respect to $\ba$ is given by $\bbij=h'(t)\baij$, so $h'(t)=c$ and hence $h(t)=ct+d$ for some constant number $d$.
\end{proof}

By Theorem \ref{main} and Proposition \ref{c}, one can determine Einstein square metrics.
Note that if the dimension of the manifold is three or four, then $\breve{\alpha}$ in Proposition \ref{c} must have constant sectional curvature $\breve{K}=1$. Thus $\bar{\alpha}$ is flat and
$F$ must have vanishing flag curvature $K=0$.  Namely, an Einstein square metric in dimension $\leq 4$ must be locally isometric to the metric in (\ref{Bmetric}).

\noindent Zhongmin Shen\\
Department of Mathematical Sciences, Indiana University-Purdue University, Indianapolis, IN 46202-3216, USA\\
zshen@math.iupui.edu
\newline
\newline
\newline
\noindent Changtao Yu\\
School of Mathematical Sciences, South China Normal
University, Guangzhou, 510631, P.R. China\\
aizhenli@gmail.com
\end{document}